\documentclass[onefignum,onetabnum]{siamart190516}

\usepackage{lipsum}
\usepackage{amsfonts}
\usepackage{graphicx}
\usepackage{epstopdf}
\usepackage{tikz,pgfplots}
\ifpdf
  \DeclareGraphicsExtensions{.eps,.pdf,.png,.jpg}
\else
  \DeclareGraphicsExtensions{.eps}
\fi

\newcommand{\matV}{F}
\newcommand{\matW}{G}

\newcommand{\fromkarl}[1]{\textcolor{red}{Karl Comment: ``#1''}}

\newsiamremark{remark}{Remark}
\newsiamremark{hypothesis}{Hypothesis}
\crefname{hypothesis}{Hypothesis}{Hypotheses}
\newsiamthm{claim}{Claim}

\headers{Shift-and-invert for singular pencil}{K. Meerbergen and Z. Wang}

\title{The shift-and-invert Arnoldi method for singular matrix pencils}
\author{Karl Meerbergen\thanks{Department of Computer Science, KU Leuven, Belgium,
    (\email{Karl.Meerbergen@cs.kuleuven.be}.)}
\and Zhijun Wang\thanks{School of Basic Education, Dalian Polytechnic University, P.R.China,
    (\email{wang.zjun.407@gmail.com}.)}
}

\usepackage{amsopn}

\ifpdf
\hypersetup{
  pdftitle={The shift inverse Arnoldi method for singular matrix pencils},
  pdfauthor={Karl Meerbergen and Zhijun Wang}
}
\fi

\usepackage{cite}
\usepackage{amsmath}
\usepackage{amssymb}
\usepackage{bbold}
\usepackage{mathtools}
\usepackage[ruled,boxed,linesnumbered,algo2e]{algorithm2e}
\usepackage{color,xcolor}
\usepackage{multirow}
\usepackage{nicematrix}

\newsiamremark{example}{Example}
\usepackage[ruled,boxed,linesnumbered,algo2e]{algorithm2e}
\begin{document}

\maketitle

\begin{abstract}
A popular method for solving large sparse regular eigenvalue problem is the 
shift-and-invert Arnoldi method.
This paper aims to use the method for large sparse singular pencils.
In three recent papers, {\em Hochstenbach, Mehl, and Plestenjak, 2019, 2023, and 2024}, propose regularization of the singular pencil, using randomly chosen regularization matrices.
We propose sparse regularization matrices obtained from the pivoting sequence of a sparse LU factorization.
As a side effect, the LU factorization often is rank revealing, which facilitates finding a regularization.
Numerical examples illustrate that the LU factorization mostly detects the normal rank and finds a suitable sparse regularization.
A rank correction method is proposed for the cases where the normal rank is not determined correctly.
For full rank rectangular eigenvalue problems, the pivoting sequence of existing sparse direct system solvers can be used.
We compare with randomized regularization methods:  preservation of sparsity is beneficial for  performance, and often, the accuracy of the eigenvalue solver.
\end{abstract}


\section{Introduction}

We consider the generalized eigenvalue problem

\begin{equation}\label{eq:GEP}
A{\bf x}=\lambda B{\bf x},
\end{equation}

\noindent where $A,B\in\mathbb{C}^{n\times n}$. If matrix pencil $A-\lambda B$ is singular, i.e., $\mathrm{det}(A-\lambda B)\equiv 0$, \eqref{eq:GEP} is called a \emph{singular eigenvalue problem}. In this case, the classical definition of eigenvalues becomes inappropriate since any scalar $\lambda_0\in\mathbb{C}$ satisfies $\mathrm{det}(A-\lambda_0 B)=0$. Therefore, the meaningful (finite) \emph{regular eigenvalues} \cite{gantmacher1959}, are defined as $\lambda_0\in\mathbb{C}$ such that $\mathrm{rank}(A-\lambda_0 B)<\mathrm{nrank}(A-\lambda B)$, where $\mathrm{nrank}(A-\lambda B):=\max_{\alpha,\beta\in\mathbb{C}}\mathrm{rank}(\alpha A-\beta B)$ is called the \emph{normal rank} of the matrix pencil $A-\lambda B$.
According to the definition, we say that there is a (regular) eigenvalue at infinity if $\mathrm{rank}(B)<\mathrm{nrank}(A-\lambda B)$.

Singular pencils often arise from multiparameter eigenvalue problems or multivariate polynomial or rational eigenvalue problems.
Rectangular eigenvalue problems can also be reformulated as square singular eigenvalue problems.
Examples of singular pencils are:
updating a finite element model to measured data (\S\ref{sec:update} and \cite{COTTIN2001}),  finding multiple eigenvalues of a parametric eigenvalue problem (\S\ref{sec:double} and \cite{Elias_DoubleEig}), and rational nonlinear eigenvalue problems (\S\ref{sec:nonlinear} and \cite{Claes2022}).
Note that, for these applications the normal rank is known from the structure of the problem, which facilitates our task.

For `small' scale and dense singular problems, the staircase method \cite{VANDOOREN1979103,byers2007,muhivc2014} may be used.
As for the regular case, the QZ method is generally reliable as well. However, these methods can be quite time and memory consuming for large-scale (sparse) problems.
Recently, a rank-completing method for singular problems was proposed in \cite{Hochstenbach2019}, \cite{Hochstenbach2023a}. 
A great advantage of this method is that it provides an efficient way to separate the true eigenvalues from the spurious ones. 
The main idea is that, singular matrix pencil $A-\lambda B$ can be transformed to 1) a regular perturbed pencil
\begin{equation}\label{eq:perturbed}
A - \lambda B + \matW (T_A - \lambda T_B) \matV^*
\end{equation}
where $\matW,\matV\in\mathbb{C}^{(n-k)\times n}$ are `random' orthogonal matrices, $T_A,T_B\in\mathbb{C}^{(n-k)\times(n-k)}$ are `random' diagonal matrices, $k=\mathrm{nrank}(A-\lambda B)$, or 2) a regular augmented pencil
\begin{equation}\label{eq:augmented}
\begin{bmatrix} A - \lambda B & \matW (T_A - \lambda T_B) \\ (S_A-\lambda S_B) \matV^* & 0\end{bmatrix}
\end{equation}
where $S_A,S_B\in\mathbb{C}^{(n-k)\times(n-k)}$ are random diagonal matrices, or 3) a regular projected pencil
\begin{equation}\label{eq:projected}
    \matW_\perp^*A\matV_\perp-\lambda \matW_\perp^*B\matV_\perp
\end{equation}
where $\matW_\perp,\matV\perp\in\mathbb{C}^{n\times k}$ are random orthogonal matrices. 
It is proved that, for generic $\matW,\matV,T_A,T_B,S_A,S_B,\matW_\perp,\matV_\perp$, the perturbed pencil \eqref{eq:perturbed}, the augmented pencil \eqref{eq:augmented}, and the projected pencil \eqref{eq:projected} are regular and include all regular eigenvalues of the original pencil $A-\lambda B$ \cite{Hochstenbach2019}.
We do not consider the perturbed pencil any further, because we did not experience an advantage.
Important to mention is that the pencils \eqref{eq:perturbed} and \eqref{eq:augmented} include prescribed eigenvalues determined by $T_A - \lambda T_B$ (and $S_A-\lambda S_B$), that can be freely chosen. 

In this paper, we explore numerical methods for solving large-scale singular eigenvalue problems, for which the isolation of the singular part is often not feasible, or may destroy the sparse structure of matrices $A$ and $B$. 
There exist plenty of numerical methods for large-scale regular eigenvalue problem \eqref{eq:GEP}. Here, we focus on the shift-and-invert Arnoldi method and its two-sided variant \cite{ruhe83}.
This method builds a Krylov space of the shift-and-invert matrix $(A-\sigma B)^{-1}B$. 
For the singular generalized eigenvalue problem, there is no $\sigma\in\mathbb{C}$ so that $A-\sigma B$ is invertible. The approach consists of applying the shift-and-invert Arnoldi method to a regularized problem such as \eqref{eq:augmented}, or \eqref{eq:projected}, with prescribed (spurious) eigenvalues at infinity, since this is usually an unwanted eigenvalue.

The use of dense $\matV$ and $\matW$ (and alternatively, $\matV_\perp$ and $\matW_\perp$) should be avoided for large scale problems.
We present a strategy using an LU factorization, that computes these matrices, and tries to determine the normal rank.
Rank revealing LU factorization exists using global maximum volume pivoting \cite{MIRANIAN20031} \cite{pan2000}.
A sparse QR factorization may be used instead \cite{pierce97}.
The development of a sparse factorization is outside the scope of the paper.
Therefore, we decided to use partial, rook and complete pivoting, which all three appear to work quite nicely for the problems solved.
Note that for regular eigenvalue problems, it is usually assumed that the shifted matrix has a good condition number to reduce rounding errors corrupting the Arnoldi recurrence relation \cite{meer00c}.
Therefore, we expect that there is clear gap between the $k$ largest singular values and the small singular values, which facilitates rank detection.
In addition, the normal rank is often known by the source of the problem. 
We propose a rank correction scheme for this case.

There is a related method to reduce the singular pencil to a problem of order $k$ by randomly combining rows and columns \cite{Hochstenbach2023a} when the normal rank is known.
The method described in this paper presents another way by selecting $k$ rows and columns, from the pivot selection of the LU factorization of $A-\sigma B$.
We compare both approaches numerically.

The paper proposes a concept rather than a complete solution: sparse linear system solvers that remove or add rows and columns do not exist and are a major challenge in software development.
Also, reliable pivoting schemes for singular matrices are currently missing.
However, for rectangular problems of full rank, we show how existing sparse direct solvers can be used.

The remainder of the paper is organized as follows. In Section~\ref{sec:prelim}, we introduce the background information and the spectral properties of the regularized pencils. Section~\ref{sec:arnoldi} presents the Arnoldi method for the regularized pencils.
Section~\ref{sec:lu} presents
a sparse LU decomposition that determines $F$ and $G$. In Section~\ref{sec:rec}, we extend our method to rectangular problems. Numerical examples are presented in Section~\ref{sec:examples} to illustrate the proposed method and to compare with randomized methods. Finally, Section~7, contains the main conclusions of the work.

\section{Two regularization pencils}\label{sec:prelim}

In this section, we first introduce the Kronecker canonical form for singular matrix pencils. Then, two regularized eigenvalue problems are presented and the corresponding spectral properties are analyzed.


For matrix pencil $A-\lambda B$ of size $n\times m$, Kronecker's theory for singular pencils \cite[Chapter XII.4]{gantmacher1959} shows that there exist matrices $P$ and $Q$ of orders $n\times n$ and $m\times m$ respectively, such that
\begin{equation*}
P(A-\lambda B)Q=\mathrm{diag}(\lambda I-J,\lambda N-I, L_{m_1}(\lambda),\ldots,L_{m_k}(\lambda), L_{n_1}(\lambda)^\top,\ldots,L_{n_l}(\lambda)^\top),
\end{equation*}
\noindent where $I$ is the identity matrix, $J$ is a Jordan matrix, $N$ is a nilpotent Jordan matrix, $L_i(\lambda)=[{\bf 0}\ I_i]-\lambda[I_i\ {\bf 0}]$ is of size $i\times(i+1)$.

This decomposition is called the Kronecker canonical form of matrix pencil $A-\lambda B$, where $\mathrm{diag}(\lambda I-J,\lambda N-I)$ is the regular part, including both a finite eigenvalue block and an infinite eigenvalue block, and $\mathrm{diag}(L_{m_1}(\lambda),\ldots,L_{m_k}(\lambda), L_{n_1}(\lambda)^\top,\ldots,L_{n_l}(\lambda)^\top)$ is the singular part, where the ``spurious'' eigenvalues come from.
Usually only the regular part (in particular, the finite eigenvalues block) is of interest to applications. 

We have a right eigenvector $x$ and left eigenvector $y$ associated with $\lambda$ iff the rank of $A-\lambda B$ is less than the normal rank, $k$.
This implies that $x$ and $y$ lie in a subspace of dimension $n-k+1$ (for a simple eigenvalue). Uniqueness can be introduced by restricting the eigenvector to a subspace of dimension $k$, chosen so that the intersection with nullspace of $A-\lambda B$ is the true eigenspace \cite{Hochstenbach2019} \cite{Hochstenbach2023a}.
This restriction can be written in the form $\matV^* x=0$ and $\matW^* y=0$ with $\matV$ and $\matW$ of rank $n-k$.
\begin{example}\label{ex:kronecker}
Let
\begin{equation}
A-\lambda B=
\begin{bmatrix}
\lambda-1 & & & \\
& -\lambda & 1 &\\
& & 0 & -\lambda\\
& & & 1
\end{bmatrix},
\end{equation}
which is of normal rank $3$.
For any $\lambda_0\in\mathbb{C}$, $A-\lambda_0 B$ has null vector $(0,1,\lambda_0,0)^T$.
The pencil has one eigenvalue $\lambda=1$, for which
the nullspace is spanned by the vectors $\{(0,1,1,0)^*,(1,0,0,0)^*\}$.
The eigenvector $x$ becomes unique by adding the constraint $\matW^*{x}={0}$ for a well chosen vector $\matW\in\mathbb{C}^n$.
For other $\lambda\neq1$, with a good choice of $\matW$,
the nullvectors are eliminated since they live in a one-dimensional subspace, which is, by the constraint, reduced to dimension zero.
A similar reasoning is used for the left eigenspace i.e., all $y\in\mathbb{C}^n$, with $(A-B)^* y=0$, by adding the constraint $F^* y=0$ for a well chosen vector $F\in\mathbb{C}^n$.
\end{example}





We use two regular eigenvalue problems for imposing this constraint. 
Let $\matV,\matW\in\mathbb{C}^{n\times(n-k)}$:
\begin{enumerate}
    \item the augmented pencil is defined as
    \begin{equation}\label{eq:AGEP}\tag{AGEP}
\begin{bmatrix}
A & \matW\\
\matV^* & 0
\end{bmatrix} - \lambda
\begin{bmatrix}
B & 0\\
0 & 0
\end{bmatrix}.
\end{equation}

\item Let $[\matV, \matV_\perp]$ and $[\matW, \matW_\perp]$ be square and full rank. 
The projected (regular) eigenvalue problem becomes
\begin{equation}\label{eq:PGEP}\tag{PGEP}
\matW_\perp^* A \matV_\perp -\lambda \matW_\perp^* B \matV_\perp.
\end{equation}
\end{enumerate}

%
We exploit a strong connection between \eqref{eq:PGEP} and \eqref{eq:AGEP} \cite[Proposition 4.1 \& Proposition 5.1]{Hochstenbach2023a}.
This connection is expressed in the following theorem.

\begin{theorem}
Let $\matV$ and $\matW$ have full rank.
Let $F^\dagger,G^\dagger\in\mathbb{C}^{(n-k)\times n}$ be such that $F^\dagger F=G^\dagger G=I$
and $F_\perp^*F=0$, and $G_\perp^*G=0$.

Then, \eqref{eq:AGEP} is regular, iff \eqref{eq:PGEP} is regular.
In addition,\begin{enumerate}
\item \eqref{eq:AGEP} and \eqref{eq:PGEP} have the same finite eigenvalues.
\item Right and left eigenvectors of \eqref{eq:AGEP} have the form
\[
\begin{pmatrix}
    z \\ -\matW^\dagger (A-\lambda B) z
\end{pmatrix}\quad,\quad \begin{pmatrix}
    w \\ -\matV^\dagger(A-\lambda B)^* w
\end{pmatrix},
\]
respectively, with $\matV^* z = \matW^* w = 0$, and the right and left eigenvectors of \eqref{eq:PGEP} are
$\matV_\perp^* z$ and $\matW_\perp^* w$, respectively.

\item $\lambda$ is a regular eigenvalue of \eqref{eq:GEP}, iff $\matW^\dagger (A-\lambda B) z=\matV^\dagger(A-\lambda B)^* w=0$.

\item \eqref{eq:AGEP} has an infinite eigenvalue with algebraic multiplicity at least $2n-2k$.
\end{enumerate}
\end{theorem}
\begin{proof}
Multiplying \eqref{eq:AGEP} on the left and the right with, respectively,
\[
\begin{bmatrix} \matW^\dagger & 0 \\ \matW_\perp^* & 0 \\ 0 & I \end{bmatrix}\quad\text{and}\quad
\begin{bmatrix} \matV & \matV_\perp & 0 \\ 0 & 0 & I \end{bmatrix},
\]
we obtain
\begin{equation}\label{eq:augmented-simple}
\begin{bNiceMatrix}[first-col,first-row]
& n-k & k & n-k \\
n-k & A_{1,1}-\lambda B_{1,1} & A_{1,2}-\lambda B_{1,2} & \alpha I \\ 
k & A_{2,1}-\lambda B_{2,1} & A_{2,2}-\lambda B_{2,2} & 0 \\ 
n-k & \alpha I & 0 & 0
\end{bNiceMatrix} \begin{pmatrix} \widetilde x_{1,1} \\ \widetilde{x}_{2,1} \\ \widetilde{x}_2 \end{pmatrix} = 0.
\end{equation}
and \eqref{eq:PGEP} becomes $(A_{2,2}-\lambda B_{2,2}) \widetilde{x}_{2,1}=0$.
From the regularity of \eqref{eq:AGEP}, we obtain the regularity of \eqref{eq:PGEP}. 

The infinite eigenvalue of \eqref{eq:AGEP} is associated with eigenvectors of the form $[0,z^T]^T$ for arbitrary $z\in\mathbb{C}^{n-k}$ and of the form $[z^T,0]$ where $Bz=0$, i.e., $z$ in the null space of $B$. Therefore, the algebraic multiplicity, is at least $2n-2k$, which proves statement~4. 


For finite eigenvalues, the eigenvectors must have the form
\begin{equation}\label{eq:SI-AGEP-eigvec}
    \begin{pmatrix} 0 \\ \widetilde{x}_{2,1} \\ -(A_{1,2}-\lambda B_{1,2})\widetilde{x}_{2,1} \end{pmatrix}. 
\end{equation}
We have a similar conclusion for the left eigenvector. This proves statement~2.

All finite eigenvalues are concentrated in $A_{2,2}-\lambda B_{2,2}$.
Statement~3 follows from $(A_{1,2}-\lambda B_{1,2})\widetilde{x}_{2,1}=0$ and $\widetilde{y}_{2,1}^*(A_{2,1}-\lambda B_{2,1})=0$.
%
%
%
%
%
%
%
\end{proof}
Note that the last block of the eigenvector is only zero when the residual $(A-\lambda B)z=0$.
This links the regularized problem to the singular pencil. It also provides a test to verify if an eigenvector of \eqref{eq:AGEP} is an eigenvector of \eqref{eq:GEP}.

\begin{example}
Recall Example~\ref{ex:kronecker}.
A suitable augmented pencil is
\begin{equation}
\left[\begin{array}{c|c}
    A - \lambda B & \matW \\\hline \matV^* & 0
\end{array}\right] = \left[\begin{array}{cccc|c}
\lambda-1 & & & & 0\\
& -\lambda & 1 & & 0\\
& & 0 & -\lambda & 0\\
& & & 0 & 1 \\\hline
0 & 0 & 1 & 0 & 0
\end{array}\right],
\end{equation}
with regular eigenvalue $1$.
There is a double spurious eigenvalue $0$ and a double eigenvalue at $\infty$.
For this choice of $\matV$ and $\matW$, \eqref{eq:PGEP} becomes
\[
\begin{bmatrix}
\lambda-1 & & \\
& -\lambda & \\
& & -\lambda \\
\end{bmatrix},\quad\text{with}\quad F_\perp = {\begin{bmatrix}
    1 & 0 & 0 \\ 0 & 1 & 0 \\ 0 & 0 & 0 \\ 0 & 0 & 1
\end{bmatrix}},\quad G_\perp={\begin{bmatrix}
    1 & 0 & 0 \\ 0 & 1 & 0 \\ 0 & 0 & 1 \\ 0 & 0 & 0
\end{bmatrix}},
\]
which is regular and has a regular eigenvalue $1$ and a double spurious eigenvalue $0$.
\end{example}

\section{Arnoldi method}\label{sec:arnoldi}

We aim to use the shift-and-invert Arnoldi method with shift $\sigma$ \cite{saad92}, i.e.,
the method is applied to the shift-and-invert matrix $S = (A-\sigma B)^{-1} B$.
The method computes the Arnoldi factorization
\[
S V_m = V_m H_m + r_m e_m^T
\]
with $V_m=[v_1,v_2,\ldots,v_m]$ the matrix holding the $m$ Arnoldi vectors, $H_m$ an upper Hessenberg matrix and $r_m$ the residual vector after $m$ iterations.
The Krylov vectors satisfy the orthogonality property $\langle v_i,v_j\rangle=0$, for $i\neq j$, where $\langle v,w\rangle$ denotes the innerproduct used.
%
The eigenvalues of $H_m$ are called Ritz values and are usually good approximations to the dominant eigenvalues of $S$.

\begin{theorem}[Equivalence of Arnoldi for augmented and projected problem]\label{th:arnoldi-augmented-projected}
Let $S^{(A)}$ be the shift-and-invert matrix with shift $\sigma$ for \eqref{eq:AGEP} and $S^{(P)}$ be the shift-and-invert matrix with shift $\sigma$ for \eqref{eq:PGEP}.
Let $m$ iterations of the Arnoldi method on $S^{(P)}$ with starting vector $v_1\in\mathbb{C}^{k}$ and Euclidean innerproduct $\langle v,w\rangle=w^*v$ produce vectors $V_{m+1}^{(P)}$ and Hessenberg matrix $H_m^{(P)}$. 
Let $m$ iterations of the Arnoldi method on \eqref{eq:AGEP} with starting vector
\[
v_1^{(A)}=\begin{pNiceMatrix}[last-col]
\matV_\perp v_1 & n \\ v_2 & n-k \end{pNiceMatrix},
\]
with $v_2$ arbitrary, and the innerproduct
\begin{equation}\label{eq:innerproductAGEP}
\langle v,w\rangle = w^* \begin{bmatrix} I_n & 0 \\ 0 & 0_{n-k} \end{bmatrix} v
\end{equation}
compute vectors $V_{m+1}^{(A)}$ and Hessenberg matrix $H_m^{(A)}$.

Then, $\matV_\perp^*V_{m+1}^{(A)}=V_{m+1}^{(P)}$ and $H_m^{(A)}=H_m^{(P)}$.
\end{theorem}
\begin{proof}
Let us choose $\sigma=0$ for ease of notation. Then $(A-\sigma B)^{-1}B=A^{-1}B$.
With $\widetilde{S}=(G_\perp^* AF_\perp)^{-1} G_\perp B$, the shift-and-invert matrices are
\begin{eqnarray*}
    S^{(P)} & = & A_{2,2}^{-1} B_{2,2} = (G_\perp^* AF_\perp)^{-1} (G_\perp B F_\perp) = \widetilde{S} F_\perp \\
    S^{(A)} & = & \begin{bmatrix} \matV & \matV_\perp & 0 \\ 0 & 0 & I \end{bmatrix} \begin{bmatrix} 0 & 0 & 0 \\ A_{2,2}^{-1} B_{2,1} & A_{2,2}^{-1} B_{2,2} & 0 \\ \star & \star & 0 \end{bmatrix} \begin{bmatrix} \matV^\dagger & 0 \\ \matV_\perp^* & 0 \\ 0 & I \end{bmatrix} \\
    & = & \begin{bmatrix}
        \matV_\perp (\matW_\perp^* A \matV_\perp)^{-1} G_\perp^* B  & 0 \\ \star & 0 
    \end{bmatrix} = \begin{bmatrix}
        \matV_\perp \widetilde{S} & 0 \\ \star & 0
    \end{bmatrix}.
\end{eqnarray*}
Let $M$ be defined by
\[
\begin{bmatrix}
    \matV_\perp \matV_\perp & 0 \\ 0 & I_{n-k}
\end{bmatrix}.
\]
Decompose $z=[z_1^T,z_2^T]^T$ with $z_1\in\mathbb{C}^n$ and $\matV_\perp \matV_\perp^* z_1=z_1$.
Also, we have $\widetilde{S} \matV_\perp \matV_\perp^*=S^{(P)} \matV_\perp^*$.
Then, we have
\[
S^{(A)} z = \begin{pmatrix}
    \matV_\perp \widetilde{S} \matV_\perp \matV_\perp^* z_1 \\ \widetilde{z}_2
\end{pmatrix} = \begin{pmatrix}
    \matV_\perp S^{(P)} \matV_\perp^* z_1 \\ \widetilde{z}_2
\end{pmatrix}.
\]

Let the Arnoldi factorization for $S^{(A)}$ be
\begin{equation}
\label{eq:arnoldi-proof-PGEP}    
S^{(A)} V_m = V_m H_m + r_m e_m^T.
\end{equation}
As a consequence, if $v_1=Mv_1$, all vectors in the Krylov space satisfy $V_m M =V_m$.
So, we have
\begin{equation}\label{eq:Arn-proof}    
M S^{(A)} V_m = M V_m H_m + M r_m e_m^T.
\end{equation}
Decompose
\[
V_{m} = \begin{pNiceMatrix}[last-col]
    V_{m}^{(1)} & n \\ V_{m}^{(2)} &  n-k
\end{pNiceMatrix}.
\]
Rewrite \eqref{eq:Arn-proof} as
\[
\begin{bmatrix}
\matV_\perp & 0 \\ 0 & I
\end{bmatrix} \left( \begin{bmatrix} S^{(P)} \matV_\perp^* V_m^{(1)} \\ \star \end{bmatrix} - \begin{bmatrix}
    \matV_\perp^* V_m^{(1)} H_m + \matV_\perp^* r_{m}^{(1)} e_m^T \\ \star
\end{bmatrix}\right) = 0.
\]
If we take the first block row, we find the Arnoldi recurrence for $S^{(P)}$ with iteration vectors $W_m = \matV_\perp^* V_m^{(1)}$.
Finally, let $H_m$ be computed using the innerproduct \eqref{eq:innerproductAGEP}.
Then
\[
H_m= \begin{bmatrix}
    (V_m^{(1)}) \\ 0 
\end{bmatrix}^* S^{(A)} \begin{bmatrix}
    V_m^{(1)} \\ V_m^{(2)}
\end{bmatrix} = (V_m^{(1)})^* \matV_\perp S^{(P)} \matV_\perp^* V_m^{(1)} = W_m^* S^{(P)} W_m.
\]
This shows the connection between the two Arnoldi factorizations.
\end{proof}
We must ensure that the starting vector of the Arnoldi method on $S^{(A)}$ has the proper form.
This can be achieved by multiplying an arbitrary vector by $S^{(A)}$ or using an implicit restart with zero shift instead \cite{Meerbergen1997}.
The choice of the starting vector combined with the special inner product remove the impact of the infinite eigenvalues introduced by augmentation.
In order to identify true eigenvalues, both left and right eigenvectors have to be computed.
Therefore we use the two-sided Arnoldi method:
\begin{enumerate}
\item Perform Arnoldi on \eqref{eq:AGEP} (resp. \eqref{eq:PGEP}). We obtain Krylov vectors $V_m$.
\item Perform Arnoldi on the transpose of \eqref{eq:AGEP} (resp. \eqref{eq:PGEP}). We obtain Krylov vectors $W_m$.
\item Compute the projected problem
\[
\widehat{A}-\lambda\widehat{B} = W_m^* (\mathbf{A}-\lambda\mathbf{B})V_m 
\]
with $\mathbf{A}-\lambda\mathbf{B}$ referring to \eqref{eq:AGEP} (resp. \eqref{eq:PGEP}).
\item Compute eigentriplets ($\lambda_i,\hat{\bf x}_i,\hat{\bf y}_i$) of $\widehat{A}-\lambda \widehat{B}$ and Ritz triplet ($\lambda_i,{\bf x}_i=V_m\hat{\bf x}_i,{\bf y}_i=W_m\hat{\bf y}_i$) of \eqref{eq:AGEP} (resp. \eqref{eq:PGEP}).
\end{enumerate}

\begin{theorem}
    The two sided Arnoldi method for \eqref{eq:AGEP} and \eqref{eq:PGEP} produce the same $\widehat{A}$ and $\widehat{B}$, when the respective starting vectors are chosen as in Theorem~\ref{th:arnoldi-augmented-projected}.
\end{theorem}
\begin{proof}
Let $V_{m+1}$ denote the Krylov vectors obtained by Arnoldi applied to \eqref{eq:AGEP} and $W_{m+1}$ to the transpose of \eqref{eq:AGEP} with starting vectors as in Theorem~\ref{th:arnoldi-augmented-projected}.
Then
\[
\begin{bmatrix}\matV\\0\end{bmatrix} ^* V_{m+1}=0\quad\text{and}\quad  W_{m+1}^*\begin{bmatrix}
    \matW \\ 0
\end{bmatrix} =0,
\]
in other words,
\[
\begin{bmatrix}
    \matV_\perp \matV_\perp^* & 0 \\ 0 & I
\end{bmatrix} V_{m+1} = V_{m+1}\quad,\quad\begin{bmatrix}
    \matW_\perp \matW_\perp^* & 0 \\ 0 & I
\end{bmatrix} W_{m+1} = W_{m+1}.
\]
Decompose
\[
V_{m} = \begin{pNiceMatrix}[last-col]
    V_{m}^{(1)} & n \\ V_{m}^{(2)} &  n-k
\end{pNiceMatrix}
\quad,\quad
W_{m} = \begin{pNiceMatrix}[last-col]
    W_{m}^{(1)} & n \\ W_{m}^{(2)} &  n-k
\end{pNiceMatrix}
\]
with $V_{m}^{(1)},W_m^{(1)}\in\mathbb{C}^{n\times(m)}$ and $V_m^{(2)},W_m^{(2)}\in\mathbb{C}^{(n-k)\times(m)}$. 
Substitute the decomposition into the projection of \eqref{eq:AGEP}, we have
\begin{eqnarray*}
W_m^* \begin{bmatrix}
    B & 0 \\ 0 & 0
\end{bmatrix} V_m & = & (\matW_\perp^* W_m)^{(1)*} (\matW_\perp^* B \matV_\perp) (\matV_\perp^* V_m)^{(1)} \\
W_m^* \begin{bmatrix}
    A & \matW \\ \matV^* & 0
\end{bmatrix} V_m & = & (\matW_\perp^* W_m)^{(1)*} (\matW_\perp^* A \matV_\perp) (\matV_\perp^* V_m)^{(1)},
\end{eqnarray*} 
which corresponds to the projection of $\matW_\perp^* A \matV_\perp-\lambda \matW_\perp^* B \matV_\perp$ on the Krylov vectors obtained for \eqref{eq:PGEP}.
\end{proof}

\section{Creation of regularized matrices by LU factorization}\label{sec:lu}
This section is about the creation of $\matV$ and $\matW$ in the pencils \eqref{eq:AGEP}, and $\matV_\perp$ and $\matW_\perp$ in \eqref{eq:PGEP}.
The LU factorization, as explained in the following sections, adds a column to $\matV$ and $\matW$ when a zero pivot is encountered during the factorization, with the aim to lift the singularity.
In practice, a test on zero is replaced by a test on small absolute values.
In this section, we explore the LU factorization using a pivoting strategy (e.g., partial, rook \cite{POOLE2000353} or complete pivoting \cite{MIRANIAN20031}) of a given matrix $A$; $\matV$ and $\matW$ are obtained in the process as by-products.

The LU factorization comes in different flavours. 
For a nonsingular $A$, we determine $PAQ=LU$ where $L$ is lower triangular with ones on its main diagonal, $U$ is upper triangular with nonzero elements on its main diagonal, and $P$ and $Q$ are permutation matrices, i.e., $P$ and $Q$ are a perturbation of the rows and columns of the identity matrix, respectively.

\subsection{Factorization for the augmented pencil}
We consider a right-looking variant of the LU factorization, where in step $i$, the first $i-1$ columns of $L$ are computed, the first $i-1$ rows of $U$ and the the right bottom block of
$U$ has the Schur complement.

Let, for a singular matrix $A$, at step $i\leq n$ the augmented matrix be denoted as
\[
\mathbf A_i = \begin{bmatrix}
    A & \matW_i \\ \matV^*_i & 0
\end{bmatrix} \in\mathbb{C}^{n_i\times n_i}
\]
where $\matV_i$ and $\matW_i$ have the same dimensions $n\times (n_i-n)$ and $n_i$ is the number of rows and columns of $\mathbf A_i$.
The initial matrix $\mathbf A_0=A$, i.e., with empty $\matV_0$ and $\matW_0$ and $n_0=n$.
In step $i$, the $i$th pivot is determined.
The pivot is searched in the submatrix composed of the rows $i$ to $n_i$ and columns $i$ to $n$.

We have from the previous step $i-1$ that
\[
P_{i-1} \mathbf A_{i-1} Q_{i-1} = L_{i-1} U_{i-1} \in\mathbb{C}^{n_{i-1}\times n_{i-1}}
\]
with $P_{i-1}$ and $Q_{i-1}$ permutation matrices and
\begin{equation}\label{eq:lu-step i}
L_{i-1} = \begin{bmatrix} L_{1,1}^{({i-1})} & 0 \\ L_{2,1}^{{(i-1)}} & I \end{bmatrix} \quad
U_{i-1} = \begin{bmatrix} U_{1,1}^{({i-1})} & U_{1,2}^{({i-1})} \\ 0 & U_{2,2}^{({i-1})} \end{bmatrix},
\end{equation}
where $U_{2,2}^{(i-1)}$ is the Schur complement of the right bottom $(n_{i-1}-(i-1))\times (n_{i-1}-(i-1))$ block of $\mathbf A_{i-1}$.
We call $L_{i-1}$ and $U_{i-1}$ that satisfy \eqref{eq:lu-step i} $L$ and $U$ factors of step ${i-1}$ of the LU factorization.
Here, matrix $L_{i-1}$ has ones on the main diagonal and is therefore invertible, $U_{1,1}^{({i-1})}$ is also invertible, $U_{2,2}^{({i-1})}$ may not be invertible.

In step $i$ there are two possibilities.
First, if there is a nonzero pivot using the pivoting strategy, then the factorization can proceed as in the usual way.
Second, if there is no nonzero pivot, we swap the columns so that the pivot column lies in column $i$, then we add a column to $\matV$ and $\matW$ to obtain
\begin{equation}\label{eq:addVW}
\begin{bmatrix}P_{i-1} & 0 \\ 0 & 1
\end{bmatrix} \begin{bmatrix} \mathbf{A}_i & g_{i} \\ f_{i}^* & 0 \end{bmatrix} \begin{bmatrix} \widetilde Q_i & 0 \\ 0 & 1
\end{bmatrix}  =
    \begin{bmatrix}
        L_{i-1} & 0 \\ 0 & 1
    \end{bmatrix}
\begin{bmatrix}
        U_{i-1} & u_{i} \\ f_i^* \widetilde{Q}_i & 0
    \end{bmatrix}
\end{equation}
with $L_{i-1} u_i=P_{i-1} g_{i}$.
Matrix $\widetilde Q_i$ accumulates the column pivots, including the pivoting at step $i$.
Note that we put a zero element below $u_i$ because of the zero at the same position in the left hand side.
At this stage no row pivoting has been performed yet.
Define
\[
Q_i=\begin{bmatrix}
    \widetilde {Q}_i & 0 \\ 0 & 1
\end{bmatrix}.
\]

There is a large amount of freedom in choosing $f_{i}$ and $g_{i}$.
Let us now choose $f_{i}\widetilde{Q}_i=\alpha e_{i}$, so that $\alpha$ becomes the pivot and we can proceed with step $i$ of the LU factorization.
The choice of $g_{i}$ is discussed further.
Now partition into
\[
\begin{bmatrix}P_{i-1} & 0 \\ 0 & 1
\end{bmatrix} \begin{bmatrix} \mathbf A_i & g_{i} \\ \alpha e_{i}^* \widetilde{Q}_i^* & 0 \end{bmatrix} \begin{bmatrix} \widetilde Q_i & 0 \\ 0 & 1
\end{bmatrix} =
    \begin{bmatrix}
        L_{1,1} & 0 & 0 \\ L_{2,1} & I & 0 \\ 0 & 0 & 1
    \end{bmatrix}
\left[\begin{array}{c|cc}
        U_{1,1} & U_{1,2} & u_1 \\\hline 0 & U_{2,2} & u_{2} \\ 0 & \alpha e_1^* & 0
    \end{array}\right]
\]
It immediately becomes clear why $f_{i}\widetilde{Q}_i=\alpha e_{i}^*$ is a good choice: the U-matrix has block upper triangular structure with the Schur complement in the bottom right corner, and the addition of a row does not require the elimination of the first ${i-1}$ elements of this new row.
We swap rows $i$ and $n_i=n_{i-1}+1$ and define $P_{i}$ that accumulates all pivoting so far into
\begin{equation}\label{eq:pivot step i}
P_{i} \begin{bmatrix} \mathbf A_i & g_{i} \\ f_{i}^* & 0 \end{bmatrix} Q_i =
    \begin{bmatrix}
        L_{1,1} & 0 & 0 \\ 0 & 1 & 0 \\ \widetilde L_{2,1} & 0 & I
    \end{bmatrix}
\begin{bmatrix}
        U_{1,1} & U_{1,2} & u_{1} \\ 0 & \alpha  e_1^* & 0 \\ 0 & \widetilde U_{2,2} & u_{2}
    \end{bmatrix}.
\end{equation}
The first column of $\widetilde U_{2,2}$ is zero, so, we have an LU factorization at step $i$.
\begin{lemma}\label{le:P}
    \[
P_i = \begin{bmatrix}
    P_{1,1}^{(i)} & P_{1,2}^{(i)} \\ P_{2,1}^{(i)} & 0
\end{bmatrix}
\]
with $P_{1,1}^{(i)}$ an $n\times n$ matrix and the other blocks with corresponding dimensions.
\end{lemma}
\begin{proof}
We prove that, after step $i$, the permutation matrix has the following structure:
\begin{equation}\label{eq:proof le}
P_{i} =  \left[\begin{array}{cc|c}
P_{1,1} & P_{1,2} & P_{1,3} \\\hline
P_{2,1} & P_{2,2} & 0 \\
P_{3,1} & P_{3,2} & 0
\end{array}\right]
\end{equation}
with $P_{1,1}$ an $i\times i$ matrix, $P_{2,2}$ an $(n-i)\times (n-i)$ matrix and $P_{3,1}$ an $(n_{i}-n)\times i$ matrix and the other matrices matching dimensions.
When $n_{i}=n$ for some $i$, then this is certainly true, since $P_{1,3}, P_{3,1}$, and $P_{3,2}$ are empty matrices.


Let us assume, it is true for $i$, then there are two situations:
the next pivot row $j$ lies between $i+1$ and $n_i$. In this case, swapping rows $i+1$ and $j$ of $P_i$ preserves the zero blocks and $P_{i+1}$ has the same structure as $P_i$;
second, the matrix is expanded and the next pivot is $n_{i+1}=n_i+1$. Then, we interchange rows $i+1$ and $n_{i+1}$ of
\[
\begin{bmatrix} P_{i} & 0 \\ 0 & 1 \end{bmatrix} =  \left[\begin{array}{cc|cc}
P_{1,1} & P_{1,2} & P_{1,3} & 0\\\hline
P_{2,1} & P_{2,2} & 0 & 0\\
P_{3,1} & P_{2,3} & 0 & 0 \\
0 & 0 & 0 & 1
\end{array}\right]
\]
which leads to the structure
\[
\left[\begin{array}{cc|cc}
P_{1,1} & P_{1,2} & P_{1,3} & 0 \\
\widetilde{P}_{2,1} & \widetilde{P}_{2,2} & 0 & 1\\\hline
\widetilde{P}_{3,1} & \widetilde{P}_{2,3} & 0 & 0 \\
\widetilde{P}_{4,1} & \widetilde{P}_{4,2} & 0 & 0
\end{array}\right].
\]
This again respects the structure that we have, which proves the lemma.
\end{proof}
 
With this procedure we can have a valid LU factorization after $n$ steps.
%
In the remainder of the section, the pivot is selected in the submatrices with row indices between $i$ and $n_i$ and column indices between $i$ and $n$.
We will show later that this is fine for rank detection.
By the fact the pivots are not searched in columns of $\matW$,
\[
Q_n = \begin{bmatrix}
    \star & 0 \\ 0 & I_{n-k}
\end{bmatrix}.
\]

After the $n$th step, we have
\begin{equation}\label{eq:PnLU}
P_{n} \begin{bmatrix}
    A & \matW \\ \matV^* & 0
\end{bmatrix} Q_n = \begin{bmatrix}
    L_{1,1}^{(n)} & 0 \\ L_{2,1}^{(n)} & I
\end{bmatrix} \begin{bmatrix}
    U_{1,1}^{(n)} & U_{1,2}^{(n)} \\ 0 & U_{2,2}^{(n)}
\end{bmatrix}
\end{equation}
where $\matW$, $U_{1,2}^{(n)}$, and $U_{2,2}^{(n)}$ have to be chosen appropriately.
The remaining pivots are now chosen from $U_{2,2}^{(n)}$.

By Lemma~\ref{le:P}, the situation is really elegant.
In this case, after step $n$,
\begin{eqnarray*}
\begin{bmatrix}
    \matW \\ 0
\end{bmatrix} & = & \begin{bmatrix} P_{1,1}^* & P_{2,1}^* \\ P_{1,2}^* & 0  \end{bmatrix}
\begin{bmatrix} L_{1,1}^{(n)} & 0 \\ L_{2,1}^{(n)} & I
\end{bmatrix} \begin{bmatrix} U_{1,2}^{(n)} \\ U_{2,2}^{(n)} \end{bmatrix}  \\
    & = &
\begin{bmatrix} P_{1,1}^* L_{1,1} U_{1,2} + P_{2,1}^* (L_{2,1} U_{1,2} + U_{2,2}) \\ P_{1,2}^* L_{1,1} U_{1,2} \end{bmatrix}
\end{eqnarray*}
A solution now is very simple: choose $U_{1,2}=0$ and $U_{2,2}=\alpha I$ so that $\matW=\alpha P_{2,1}^*$.
Note that an explicit factorization after step $n$ is not needed since $U_{2,2}=\alpha I$, which is already in factored form.
Also note that $\matW$ has full rank by construction since $P_n$ is a permutation matrix.

The obtained augmented pencil is definitely regular by construction, since $L$ and $U$ are non-singular. We still have to prove that the dimension is indeed the minimal dimension to uncover the normal rank.
\begin{lemma}\label{le:aux}
Using the partitioning of \eqref{eq:PnLU}, we have
\begin{enumerate}
    \item $P_{1,2}^* L_{1,1} = P_{1,2}^*$,
    \item $L_{1,1} P_{1,2}= P_{1,2}$, and $L_{1,2}P_{1,2}=0$
    \item $L_{1,2}P_{1,2}=0$,
    \item $P_{1,2}^* P_{1,1}=0$ and $P_{1,1}^* P_{1,2}=0$.
\end{enumerate}
\end{lemma}
\begin{proof}
At step $n$, we have the factorization
\begin{equation}\label{eq:LU-proof-1}
P \begin{bmatrix} A & \matW \\ \matV^* & 0 \end{bmatrix} Q = \begin{bmatrix} L_{1,1} & 0 \\ L_{2,1} & I \end{bmatrix} \begin{bmatrix} U_{1,1} & 0 \\ 0 & U_{2,2} \end{bmatrix}
\end{equation}
with $L_{1,1}$ and $U_{1,1}$ of dimension $n\times n$ and the other blocks with matching dimensions.

Denote all steps $i$ of a breakdown of the LU factorization, i.e., all $i$, for which a row to $\matV^*$ is added, by $i_1,\ldots,i_{n_n-n}$: in step $i_j$ row $i_j$ is swapped with row $n+j$.
From \eqref{eq:pivot step i}, we see that, after pivoting, the $i_j$th row of $L$ corresponding to such a breakdown step is zero, except for the diagonal position which is one.
Note that the columns of $P_{1,2}$ contain a value one in the $(i_j,j)$ positions, for $j=1,\ldots,n_n-n$.
Since $P_{1,2}^* L_{1,1}$ is the selection of rows of $L_{1,1}$ that corresponds to the breakdown steps in the factorization and $L_{1,1}$ has a zero row and column in such position (except the main diagonal element),
we have that $P_{1,2}^* L_{1,1}=P_{1,2}^*$.
Similarly, $L_{1,1} P_{1,2}= P_{1,2}$ and $L_{2,1}P_{1,2}=0$.
Because of the orthogonality of $P$, we have that $P_{1,2}^* P_{1,1}=0$ and $P_{1,1}^* P_{1,2}=0$.
\end{proof}
\begin{theorem}\label{th:rank AsigmaB}
Let $k$ be the rank of $A-\sigma B$.
Let $\matV$ and $\matW$ be determined as explained in this section, then the number of columns of $F$ and $G$ is $n-k$. 
\end{theorem}
\begin{proof}
Assume that $A-\sigma B$ has rank $k$.
Let $\sigma$ be zero to simplify notation.

At step $n$, we have the factorization \eqref{eq:LU-proof-1},
with $L_{1,1}$ and $U_{1,1}$ of dimension $n\times n$ and the other blocks with matching dimensions.
As in Lemma~\ref{le:aux}, denote all steps $i$ of a breakdown of the LU factorization, by $i_1,\ldots,i_{n_n-n}$.

We are applying a permutation matrix that brings the added rows $\matV^*$ back to the bottom of the matrix.
To achieve this, we define
\[
\widetilde P = \begin{bmatrix}
 Y P_{1,1}^* & Y P_{2,1}^* \\  
    P_{1,2}^* & 0
\end{bmatrix}
\]
on the left of \eqref{eq:LU-proof-1}.
Matrix $Y$ is chosen so that $\widetilde{P}$ is a permutation matrix, i.e., $Y$ itself can be any order $n$ permutation matrix.
Let
\[
Q = \begin{bmatrix} \widetilde{Q} & 0 \\ 0 & I \end{bmatrix}.
\]
Since from Lemma~\ref{le:aux}, $P_{1,1}P_{1,2}^* =0$, we have that
\begin{eqnarray*}
\widetilde P P \begin{bmatrix}
    A & \matW \\ \matV^* & 0
\end{bmatrix}Q & = & \begin{bmatrix}
    Y A\widetilde{Q} & Y \matW \widetilde{Q} \\ \matV^*\widetilde{Q} & 0
\end{bmatrix}.
\end{eqnarray*}
For the application on $L$, we have, again using Lemma~\ref{le:aux},
\begin{eqnarray*}
\widetilde P L \widetilde{P}^* & = & \begin{bmatrix}
    Y \begin{bmatrix}
        P_{1,1} \\ P_{2,1}
    \end{bmatrix}^* \begin{bmatrix}
        L_{1,1} \\ L_{2,1}
    \end{bmatrix} & Y P_{2,1}^* \\ P_{1,2}^* L_{1,1} & 0
\end{bmatrix} \begin{bmatrix} P_{1,1} Y^* & P_{1,2} \\ P_{2,1} & 0 \end{bmatrix} \\
& = & 
\begin{bmatrix}
    \widetilde L_{1,1} & \widetilde{L}_{1,2} \\ P_{1,2}^* L_{1,1} P_{1,1} Y^* & P_{1,2}^* L_{1,1} P_{1,2}
\end{bmatrix} \\
\widetilde{L}_{1,2} & = & Y \begin{bmatrix} P_{1,1} \\ P_{2,1} \end{bmatrix}^* \left(\begin{bmatrix} L_{1,1} \\ L_{2,1} \end{bmatrix} P_{1,2}\right)
    = Y \begin{bmatrix} P_{1,1} \\ P_{2,1} \end{bmatrix}^* \begin{bmatrix} P_{1,2} \\ 0 \end{bmatrix} = 0,\\
P_{1,2}^* L_{1,1} P_{1,1} Y^* & = & P_{1,2}^* P_{1,1} Y^* = 0, \\
P_{1,2}^* L_{1,1} P_{1,2} & = &  P_{1,2}^* P_{1,2} = I.
\end{eqnarray*}
We conclude that
\[
\widetilde P L \widetilde{P}^* = \begin{bmatrix} \widetilde L_{1,1} & 0 \\ 0 & I \end{bmatrix}.
\]
For the U-factor, we have
\begin{eqnarray*}
    \widetilde{P} U & = & \begin{bmatrix} Y P_{1,1}^* U_{1,1} & Y P_{2,1}^* U_{2,2} \\ P_{1,2}^* U_{1,1} & 0 \end{bmatrix} = \begin{bmatrix} \widetilde  U_{1,1} & \widetilde U_{2,2} \\ \matV^*\widetilde{Q} & 0 \end{bmatrix} .
\end{eqnarray*}
As a result, we have
\begin{equation}
\begin{bmatrix}
    Y A & Y \matW \\ \matV^* & 0
\end{bmatrix}Q = \begin{bmatrix} \widetilde L_{1,1} & 0 \\ 0 & I \end{bmatrix} \begin{bmatrix} \widetilde{U}_{1,1} & \widetilde{U}_{1,2} \\ \matV^*\widetilde{Q} & 0 \end{bmatrix}.
\end{equation}

Matrix $\widetilde{U}_{1,1}$ is a row perturbation and selection of rows of $U_{1,1}$. Since $U_{1,1}$ is upper triangular, we can choose $Y$ so that $\widetilde{U}_{1,1}$ is also upper triangular.
We conclude that $YA\widetilde{Q} = \widetilde{L}_{1,1} \widetilde {U}_{1,1}$.
Since $\widetilde{L}_{1,1}$ has rank $n$, the rank of $\widetilde{U}_{1,1}$ is the rank of $A$.
By construction, the rank of $\widetilde{U}_{1,1}$ is equal to the rank of $P_{1,1}$, which corresponds to the selection of rows with standard pivoting, i.e., without breakdown.
In other words, the rank of $P_{1,1}$ indeed corresponds to $k$.
\end{proof}

\noindent
An algorithm is given in Algorithm~\ref{alg:LU-2}.
The algorithm is for a rectangular matrix $A-\sigma B\in\mathbb{R}^{n\times m}$ with $m\leq n$, so that it can be used for \S\ref{sec:rec}.
We choose $\alpha=\|A-\sigma B\|$ using some norm, for example, the one or infinity norm 
or an estimate of the two-norm using a randomized method.

\begin{algorithm2e}[H]
\caption{LU factorization with rank detection}
\label{alg:LU-2}\scriptsize
\KwIn{matrix $A-\sigma B\in\mathbb{R}^{n\times m}$ of rank $k$ and $m\leq n$.}
\KwOut{border $\matV$ and $\matW$, permutation matrices $P$ and $Q$, lower triangular matrix $L$, and upper triangular matrix $U$, where $P,Q,L,U\in\mathbb{R}^{(n+m-k)\times (n+m-k)}$}

Set $\matV$ be an empty matrix, and let $U=A-\sigma B$, $n_0=n$, and
$P=Q=I_n$.

\For{$i=1,2,\ldots,m$}
{
   Select a row pivot $p_r$ in $\{i,\ldots,n_{i-1}\}$ and a column pivot $p_c$ in $\{i,\ldots,m\}$.\\
    \uIf{$|U_{p_r,p_c}|<\tau\alpha$}
    {
        Add a new row to $U$:
        let $U = \begin{bmatrix}
            U \\ \alpha e_i 
        \end{bmatrix}$,\\
        Let $P=\begin{bmatrix}
            P & 0 \\ 0 & 1
        \end{bmatrix},\,
        Q=\begin{bmatrix}
            Q & 0 \\ 0 & 1
        \end{bmatrix},\,
        L=\begin{bmatrix}
            L & 0\\ 0 & 1
        \end{bmatrix}$.\\
        Add the column $\alpha e_i$ to $\matV$.\\
        Let $n_i = n_{i-1}+1$ and $p_r=i_n$.
    }
    \Else{
     Let $n_{i} = n_{i-1}$.
    }

    Pivot rows $i$ and $p_r$ in $P$, $L$ and $U$.\\
    Pivot columns $i$ and $p_c$ in $Q$, $L$ and $U$ and rows $i$ and $p_c$ in $\matV$.\\
    Compute the $i$th column of $L_{i+1:n_i,i}=U_{i+1:n_i,i}/U_{i,i}$.\\
    Update: $U_{i+1:n_i,i+1:n} = U_{i+1:n_i,i+1:m} - U_{i+1:i_n,i} L_{i,i+1:m}$.
}

Let $\matW = \alpha P_{n+1:n_m,1:n_m-n}^*\in\mathbb{R}^{n\times (n_m-n)}$.

Let $U = \begin{bmatrix} U  & \begin{matrix} 0 \\ \alpha I\end{matrix}\end{bmatrix}$, where $I$ is the identity matrix of proper dimension.
\end{algorithm2e}
Instead of testing whether the pivot is zero, we use  the test in line~4, with $\tau$ a prescribed tolerance.
An exact zero is very unlikely to happen even when $A-\sigma B$ is singular, due to rounding errors.
In this case, it is possible that more columns of $\matV$ are added to the border than needed to detect the rank.
See~\ref{sec:rank correction}.

\subsection{Factorization for the projected problem}

For the projected problem \eqref{eq:PGEP}, the procedure is similar to the augmented LU factorization.
Instead of adding a border, we remove the pivoting row and column, when no suitable pivot is found.
Matrix $\matV_\perp$ is easily derived from $\matV$ as the identity matrix with rows removed that correspond to the zero pivots. The reasoning is similar for $\matW_\perp$.
In addition, the pivoting is applied to $\matV_\perp$ and
$\matW_\perp$, respectively, so that
\begin{equation}
\label{eq:LU PGEP}    
\matW_\perp^*(A-\sigma B)\matV_\perp=LU.
\end{equation}
\begin{theorem}
Let $\matW_\perp^*(A-\sigma B)\matV_\perp$ have full rank $k$.
Given the LU factorization of \eqref{eq:LU PGEP}, then, there are $F$ and $G$ so that
$[F, F_\perp]$ and $[G, G_\perp]$ are square and full rank matrices, and there is an LU factorization
\[
\widetilde{P} \begin{bmatrix} A-\sigma B & \matW \\ \matV^* & 0 \end{bmatrix} \widetilde{Q} = \begin{bmatrix}
    L & 0 \\ \star & \star
\end{bmatrix} \begin{bmatrix}
    U & \star \\ 0 & \star
\end{bmatrix}.
\]
\end{theorem}
\begin{proof}
By definition of $F_\perp$ and $G_\perp$, there are $n\times (n-k)$ matrices $\widetilde{F}$ and $\widetilde{G}$ so that
$[F_\perp, \widetilde{F}]$ and $[G_\perp,\widetilde{G}]$ are permutation matrices.
Then, we have
\[
\begin{bmatrix}
G_\perp^* \\ \widetilde{G}^*
\end{bmatrix}(A-\sigma B)\begin{bmatrix}
F_\perp & \widetilde{F}
\end{bmatrix} = \begin{bmatrix} L & 0 \\ L_R & I \end{bmatrix} \begin{bmatrix}
    U & U_R \\ 0 & S_R
\end{bmatrix}
\]
where $L$ and $U$ are full rank and $S_R$ is the (singular) Schur complement.

By adding the border with $F=\alpha\widetilde F$ and $G=\alpha \widetilde G$, we have
\begin{equation}\label{eq:LU AGEP-proof1}
\begin{bmatrix} G_\perp^* & 0 \\ \widetilde G^* & 0 \\ 0 & I \end{bmatrix}
\begin{bmatrix} A-\sigma B & {G} \\ {F}^* & 0 \end{bmatrix}
\begin{bmatrix} F_\perp^* & \widetilde F^* & 0 \\ 0 & 0 & I\end{bmatrix} =
\begin{bmatrix} L & 0 & 0 \\ L_R & I & 0 \\ 0 & 0 & I \end{bmatrix}
\begin{bmatrix} U & U_R & 0 \\ 0 & S_R & \alpha I \\ 0 & \alpha I & 0\end{bmatrix}.
\end{equation}
Now, we apply a row pivoting that swap the rows of the border with the rows that are discarded by the projected factorization.
To do that, we multiply on the left with the permutation matrix
\[
\begin{bmatrix} I_k & 0 & 0 \\ 0 & 0 & I_{n-k} \\ 0 & I_{n-k} & 0\end{bmatrix}.
\]
The right-hand-side of \eqref{eq:LU AGEP-proof1} becomes after permutation
\begin{eqnarray*}
\begin{bmatrix} L & 0 & 0 \\ 0 & I & 0 \\ L_R & 0 & I \end{bmatrix}
\begin{bmatrix} U & U_R & 0 \\ 0 & \alpha I & 0\\ 0 & S_R & \alpha I \end{bmatrix}
=  
\begin{bmatrix} L & 0 & 0 \\ 0 & I & 0 \\ L_R & \alpha^{-1} S_R & I \end{bmatrix}
\begin{bmatrix} U & U_R & 0 \\ 0 & \alpha I & 0\\ 0 & 0 & \alpha I \end{bmatrix},
\end{eqnarray*}
which is the desired factorization.
\end{proof}

This does no longer hold when the pivoting sequence is different, which can happen, since pivot elements for the factorization for \eqref{eq:AGEP} can be chosen among the discarded rows $S_R$.
The theorem does not make any assumption on the discarded pivots, i.e., pivots do not have to be zero.
Also, no assumption is made on the rank of $S_R$.

By construction, the LU factorizations for \eqref{eq:PGEP} and \eqref{eq:AGEP} proceed in the same way and produce the same numbers in the LU factors.
When a pivot is discarded, the factorization of \eqref{eq:PGEP} removes the pivot row and columns, where the factorization of \eqref{eq:AGEP} swaps this pivot row with a row in $\matV^*$. This introduces zeroes in the pivot row except for $\alpha$ as a pivot.
Because of the zeroes, the Schur complement does not depend on this pivot.
As a result, the factorization for \eqref{eq:AGEP} can be derived from the factorization for \eqref{eq:PGEP}.

\subsection{Rank correction}\label{sec:rank correction}
For the shift-and-invert transformation, it is expected that the condition number of the shifted matrix is not too high to prevent the Arnoldi method to introduce large errors in the
recurrence relation. This is usually achieved by picking the shift in between (clusters of) eigenvalues.
This means, that it is expected that the normal rank is relatively easy to uncover. If this is not so, the Arnoldi method may not be reliable.

For many problems, the normal rank is known from the structure of the problem. This information is not so easy to use in the LU factorizations that we presented.
One could choose the tolerance $\tau$ by trial and error until the uncovered rank corresponds to the normal rank.
In this section, we propose another idea.
Suppose that the detected normal rank $\widetilde k$ is smaller than the true rank: this happens when the drop tolerance is too large; see \S\ref{sec:update} for an example.

Suppose that $\widetilde\matV$ and $\widetilde\matW$ have been formed (or the orthogonal complements) and that their column sizes are larger than expected.
Let $\widetilde\matV,\widetilde\matW\in\mathbb{R}^{n\times (n-\widetilde k)}$ with $\widetilde k<k$.
Now define random matrices $Z,Y\in\mathbb{R}^{(n-\widetilde{k})\times (n-k)}$ with orthogonal complements $Z_\perp,Y_\perp\in\mathbb{C}^{(n-\widetilde{k})\times(k-\widetilde{k})}$, respectively. (In practice, we use Matlab's \texttt{randn}.)
Define $\matV=\widetilde\matV Z$ and $\matW=\widetilde\matW Y$, and the orthogonal complements
\begin{eqnarray*}
    \matV_\perp & = & \begin{bmatrix}
        \widetilde\matV_\perp & \widetilde\matV Z_\perp
    \end{bmatrix} \\
\matW_\perp & = & \begin{bmatrix}
        \widetilde\matW_\perp & \widetilde\matW Y_\perp
    \end{bmatrix}.
\end{eqnarray*}

Assume for ease of notation that $\sigma=0$. For \eqref{eq:PGEP}, the new linear system becomes
\begin{equation}\label{eq:PGEP rank correction}
\begin{bmatrix}
    \widetilde\matW_\perp^*A\widetilde\matV_\perp & \widetilde\matW_\perp^*A\widetilde\matV Z_\perp \\
    Y_\perp^* \widetilde\matW^*A\widetilde\matV_\perp & Y_\perp^* \widetilde\matW^*A\widetilde\matV Z_\perp \\
\end{bmatrix} \begin{pmatrix} x\\y\end{pmatrix} = \begin{pmatrix}
    f \\ g
\end{pmatrix}.
\end{equation}
Since we expect $\widetilde{k}\simeq k$, $y$ is low dimensional and can cheaply be solved, using
the given LU factorization $LU=\widetilde\matW_\perp^*A\widetilde\matV_\perp$.
Assuming that \eqref{eq:PGEP rank correction} is a regular linear system, its solution is computed from the block LU factorization
\[
\begin{bmatrix}
LU & 0 \\ Y_\perp^* \widetilde\matW^*A\widetilde\matV_\perp & S
\end{bmatrix}
\begin{bmatrix}
I & (LU)^{-1} \widetilde\matW_\perp^*A\widetilde\matV Z_\perp \\ 0 & I
\end{bmatrix} \begin{pmatrix} x\\y\end{pmatrix} = \begin{pmatrix} f \\ g \end{pmatrix},
\]
with Schur complement $S=Y_\perp^* \widetilde\matW^*A\widetilde\matV Z_\perp-Y_\perp^* \widetilde\matW^*A\widetilde\matV_\perp(LU)^{-1} \widetilde\matW_\perp^*A\widetilde\matV Z_\perp$.
Since $S$ is supposed small scale, the additional cost is not high.
For a linear solver, $S$ is precomputed and factorized as a dense matrix.

The LU factorization for \eqref{eq:AGEP} can be used for solving
\begin{equation}\label{eq:AGEP-rank corrected}
\begin{bmatrix}
    A & \matW & \widetilde{\matW}Y_\perp \\
    \matV^* & 0 & 0 \\
    Z_\perp^* \widetilde\matV^* & 0 & 0
\end{bmatrix}\begin{pmatrix}
    x_1 \\ x_2 \\ y
\end{pmatrix} = \begin{pmatrix} f \\ 0 \\ 0 \end{pmatrix},
\end{equation}
using an orthogonal transformation of right-hand side and solution.
Since we only want to solve a linear system with a submatrix, we could instead solve the system
\begin{equation}\label{eq:AGEP-rank corrected final}
\begin{bmatrix}
    A & \matW & 0 \\
    \matV^* & 0 & 0 \\
    Z_\perp^* \widetilde\matV^* & 0 & I
\end{bmatrix}\begin{pmatrix}
    x_1 \\ x_2 \\ y
\end{pmatrix} = \begin{pmatrix} f \\ 0 \\ 0 \end{pmatrix}.
\end{equation}
Assuming that $\widetilde{k}\simeq k$, this linear system is a (low) rank $k-\widetilde k$ change of \eqref{eq:AGEP-rank corrected} and \eqref{eq:AGEP-rank corrected final} can be solved using the Sherman-Morrison-Woodbury formula, with a small additional cost. We do not discuss this in further detail.

\section{Rectangular pencils}\label{sec:rec}

In this section, we employ the proposed method for rectangular eigenvalue problem \eqref{eq:GEP}, where $A,B\in\mathbb{R}^{n\times m}$ and $n\neq m$. Examples arise from rectangular multiparameter eigenvalue problems \cite{Shapiro2009} \cite{hochstenbach2024solvingsingulargeneralizedeigenvalue}.
Similarly to the square case, $\lambda\in\mathbb{C}$ is called an eigenvalue of the problem \eqref{eq:GEP} if $\mathrm{rank}(A-\lambda B)$ is less than the normal rank $\max_{\sigma\in\mathbb{C}}\mathrm{rank}(A-\sigma B)$.
Let us for the ease of presentation, assume $m<n$.

First, let the normal rank $k=m$.
In this case, \eqref{eq:AGEP} becomes
\[
\begin{bmatrix}\label{eq:RectGEP}
    A - \lambda B & \matW
\end{bmatrix}.
\]
Similar to the square case, after step $m$ of the LU factorization, we choose
\[
\matW=P_m^* \begin{bmatrix} 0 \\ \alpha I_k \end{bmatrix}.
\]
Partial pivoting should be sufficiently reliable for problems with $n>m$ and $k=m$, since this method is usually performed for sparse pencils of full rank.
The pencil \eqref{eq:PGEP} becomes
\[
\matW_\perp^*(A-\lambda B).
\]
The latter corresponds to a selection of rows of $A-\lambda B$.
Interestingly, we can use existing sparse solvers with row pivoting for the rectangular matrix.
The factorization obtained satisfies
$P(A-\sigma B)Q = LU$.
Selecting $\widetilde P$ and $\widetilde L$ as the first $m=k$ rows of $P$ and $L$, respectively, we have that $\widetilde P(A-\lambda B)Q$ is a square and regular eigenvalue problem and $\widetilde P(A-\sigma B)Q = \widetilde LU$.
In the numerical example in \S\ref{sec:nonlinear}, we used the built-in Octave sparse solver for \eqref{eq:PGEP}.

If $A-\sigma B$ is not full rank, we use both $\matV$ and $\matW$, this time of different dimensions to obtain a square pencil.
We perform $m$ LU factorization steps, as before.
If, after step $m$, the number of rows of $A-\sigma B$ expanded with $\matV^*$ is larger than the number of columns of $A-\sigma B$, then we add $\matW$ to make the augmented matrix square.




\section{Numerical examples}\label{sec:examples}

The numerical results were obtained using Octave version~9.2.0 on a Macbook Pro with 24 GB RAM and an Apple M3 chip running MacOS Sequoia~15.7.3.
The software is available through \url{https://gitlab.kuleuven.be/numa/public/singular_arnoldi.git}.
The LU factorization algorithms are naive implementations from the algorithms in the paper. No optimizations for performance or fill-in reducing numerical pivoting was performed \cite{duffpralet2005}.
Condition numbers were computed with Octave's command $\texttt{cond}$ or $\texttt{condest}$ for matrices larger than 1000.
The Arnoldi method used classical Gram-Schmidt orthogonalization with reorthogonalization.
Implicit restarts are computed using an explicit QR step.

We used partial row pivoting, rook pivoting \cite{POOLE2000353} and complete pivoting in our experiments. The number of row or column searches per elimination step using rook pivoting was limited to five.
Also, to further reduce the number of searches and column swaps, rook pivoting was stopped after the first column search if a suitable pivot was found, i.e., larger than the tolerance.

\subsection{Updating finite element model}\label{sec:update}

The equation of motion of a damped linear dynamic system can be expressed as
$$M\ddot{\bf q}(t)+C\dot{\bf q}(t)+K{\bf q}(t)=0,$$
where $M,C,K\in\mathbb{R}^{n\times n}$ are inertia, damping, and stiffness matrices, respectively. Let ${\bf q}(t)=e^{\omega t}{\bf x}$, it can be transformed into a quadratic eigenvalue problem
\[
(\omega^2 M+\omega C+K){\bf x}=0,
\]
where $\omega$ and ${\bf x}$ are eigenvalue and eigenvector, respectively. 
The example is a finite element model of a simple truss structure. The model, shown in Figure~\ref{figure1}, consists of 40 elements, 18 nodes, and 33 degrees of freedom (DOFs) \cite{wang2023}.
The elements are made of steel with mass density of $8000\,\mathrm{kg/m^3}$, Young's modulus of $200\,\mathrm{GPa}$, and cross section area of $20\,\mathrm{cm^2}$. A nominal finite element model $(M,C,K)$ can be constructed, where $C=a_0M+a_1K$ is the Rayleigh damping matrix with $0.5\%$ damping ratio \cite{Shadan2015}. In addition, $10\%$ loss of stiffness is applied to imitate the degradation of the materials for the 3rd, 5th, and 7th elements in Figure~\ref{elements1}. Consequently, another finite element model $(M,C^*,K^*)$ is constructed based on the degraded materials, where $C^*=a_0^*M+a_1^*K^*$.

\begin{figure}
\centering
\begin{tabular}{c}
\includegraphics[scale=0.4]{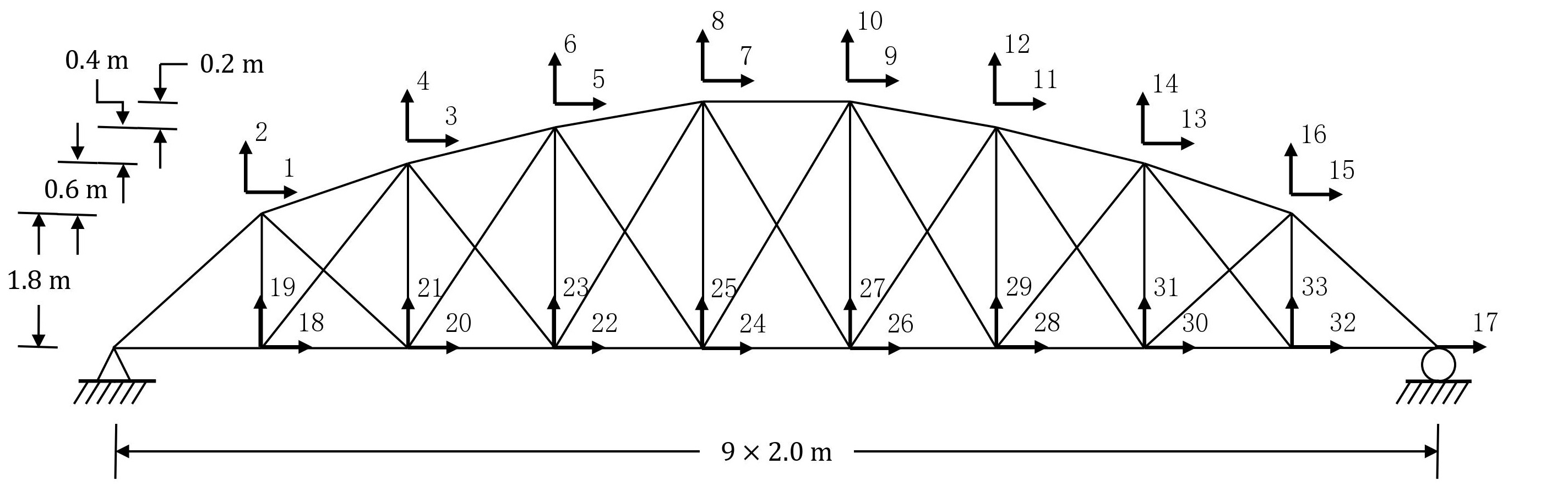}
     \\
\includegraphics[scale=0.4]{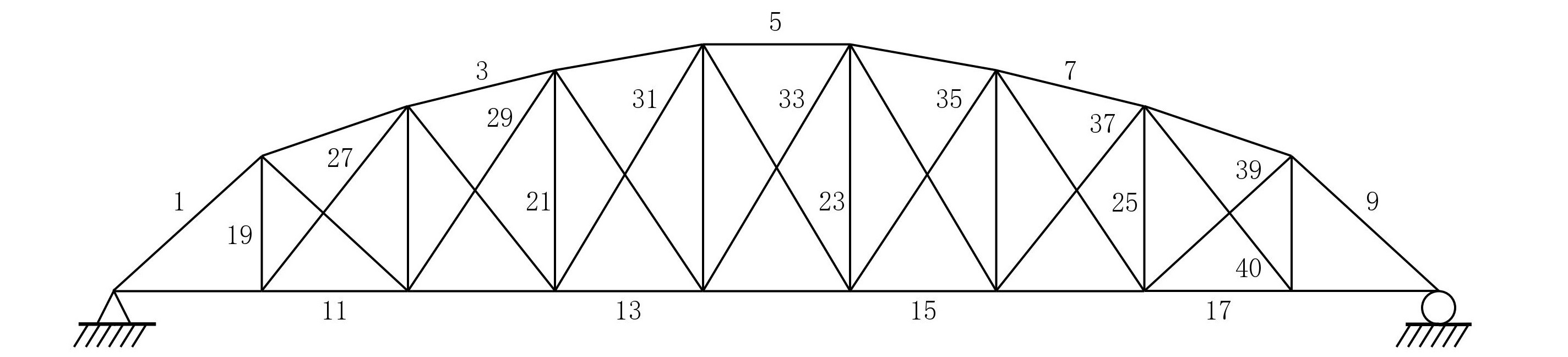} \\
    33 degrees of freedom
\end{tabular}
\caption{Truss model}
\label{elements1}\label{figure1}       
\end{figure}
\begin{figure}
    \centering
    \begin{tabular}{cc}\includegraphics[width=0.4\textwidth]{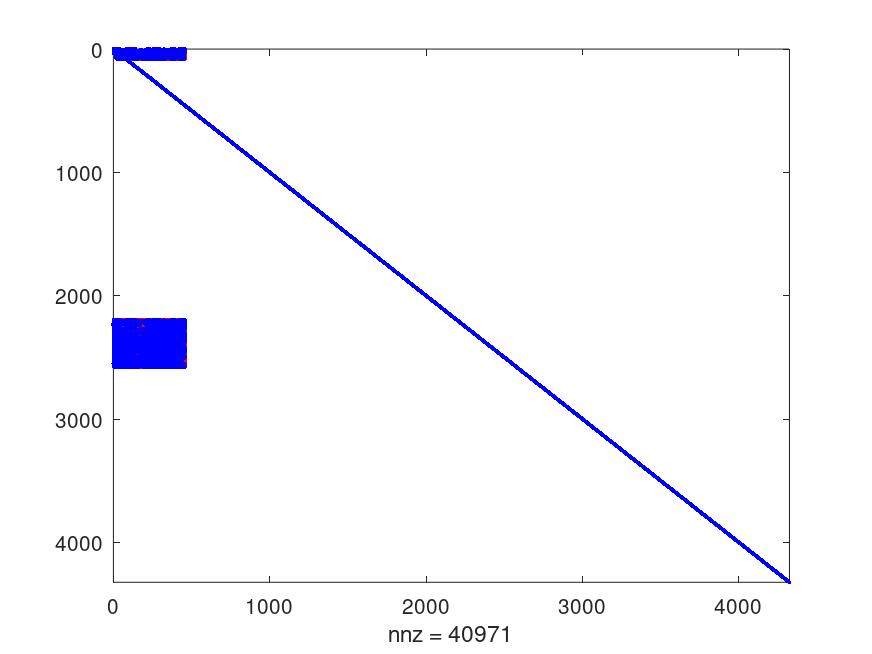} &
    \includegraphics[width=0.4\textwidth]{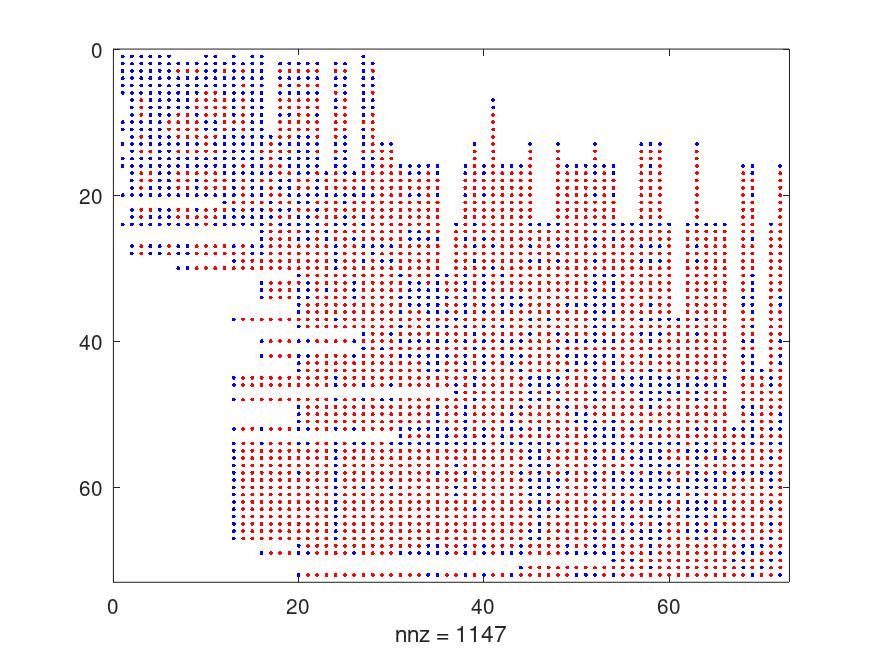} \\
    \eqref{eq:AGEP} & \eqref{eq:PGEP}
    \end{tabular}
    \caption{Sparsity patterns of pencil matrices (blue) and the fill-in by the L and U factors obtained with rook pivoting (red). Pencil matrices are permuted following the pivoting matrices to match the pattern of $L$ and $U$. The fill-in reducing ordering \texttt{colamd} was used before the LU factorization.
    If we zoom in the first 72 rows and columns of the left figure, we obtain the right figure.}
\end{figure}

The updating of the model to measured eigenvalues leads to a multiparameter eigenvalue problem \cite{COTTIN2001}, which can be transformed into a singular generalized eigenvalue problem \cite{atkinson1972}. 
We  update the stiffness of the 3rd, 5th, and 7th elements in the nominal finite element model $(M,C,K)$ using three eigenvalues measured in the experimental model $(M,C^*,K^*)$. 

We noticed that the rank of $B$ is six, so, there are at most six finite eigenvalues for the singular pencil, and only six finite eigenvalues for the regularized pencils.
Therefore, six iterations of Arnoldi are sufficient to obtain all eigenvalues, provided the starting vector lies in the invariant subspace.
This is achieved by choosing the starting vector as $((A-\sigma B^{-1}B)^2 \widetilde{v}$ with $\widetilde{v}$ randomly chosen.

First, we ran a smaller problem with 13 DOFS instead of 33, which leads to a pencil of order $n=2197$ and nrank $k=72$. We choose shift $\sigma=200$ and found the eigenvalues with residual norms of the regularized and original pencils of the order $10^{-15}$ for all methods using \eqref{eq:AGEP}, and \eqref{eq:PGEP}.
Note that without implicit restarts with zero shifts, even for \eqref{eq:PGEP}, the residual was only of the order $10^{-12}$.

As a comparison, we tried a randomized approach by forming \eqref{eq:PGEP} with randomly chosen orthonormalized $\matV_{\perp}$ and $\matW_\perp$. Note that it is important that $\matV_{\perp}$ and $\matW_\perp$ are orthonormalized, otherwise results are not good.
We found that, with our approach the condition number of $\matW_\perp^*(A-\sigma B)\matV_\perp$ was $2.1\cdot10^2$ and for the randomized approach it was
$1.1\cdot10^6$, showing at least one advantage of a deterministic way in determining pivots and projectors.
The residual norm for the right and left Ritz vectors associated with Ritz value $-130.2$ was $4.5\cdot10^{-15}$ and $5.0\cdot10^{-15}$, respectively, for our approach and $5.2\cdot10^{-12}$ and $5.0\cdot10^{-11}$ for the randomized approach, which shows a second advantage of ours.


We now illustrate rank correction for \eqref{eq:PGEP} and \eqref{eq:AGEP}.
We ran the PLU factorization with $\tau=10^{-1}$ and found that the normal rank was wrongly determined as $68$ instead of $72$ for both pencils.
We then reran the Arnoldi method with the rank correction from \S\ref{sec:rank correction}.
Tables~\ref{tab:rank correction PGEP} and~\ref{tab:rank correction AGEP} summarize the results.
\begin{table}
    \caption{Real parts of first three Ritz values for \eqref{eq:PGEP} with and without rank correction for $\tau=10^{-1}$, $\widetilde k=68$ and $k=72$.
$\rho_R=\|Ax-\lambda Bx\|_2/\|Ax\|_2$, $\rho_L=\|y^* A-\lambda y^* B\|_2/\|y^*A\|_2$.}\label{tab:rank correction PGEP}
    \begin{center}\footnotesize
    \begin{tabular}{lll|lll}
    \multicolumn{3}{c|}{Without rank correction} & \multicolumn{3}{c}{With rank correction} \\
    \multicolumn{1}{c}{$\lambda$} & \multicolumn{1}{c}{$\rho_R$} & \multicolumn{1}{c|}{$\rho_L$} & \multicolumn{1}{c}{$\lambda$} & \multicolumn{1}{c}{$\rho_R$} & \multicolumn{1}{c}{$\rho_L$} \\\hline\rule{0pt}{2.3ex}%
    $-176.67$ &  $2.3\cdot10^{-1}$ & $1.7\cdot10^{-1}$ & $-130.21$ & $2.2\cdot10^{-13}$ & $2.9\cdot10^{-13}$ \\
    $-187.94$ & $1.4\cdot10^{-1}$ & $6.6\cdot10^{-2}$ & $-180.00$ & $6.1\cdot10^{-15}$ & $9.9\cdot10^{-14}$ \\
    $-214.2$ & $1.50\cdot10^{-1}$ & $8.2\cdot10^{-2}$ & $-219.39$ & $6.2\cdot10^{-14}$ & $1.8\cdot10^{-14}$ \\
    \end{tabular}
    \end{center}
\end{table}
\begin{table}
    \caption{Real parts of first three Ritz values for \eqref{eq:AGEP} with and without rank correction for $\tau=10^{-1}$, $\widetilde k=68$ and $k=72$.
$\rho_R=\|Ax-\lambda Bx\|_2/\|Ax\|_2$, $\rho_L=\|y^* A-\lambda y^* B\|_2/\|y^*A\|_2$.
}\label{tab:rank correction AGEP}
    \begin{center}\footnotesize
    \begin{tabular}{lll|lll}
    \multicolumn{3}{c|}{Without rank correction} & \multicolumn{3}{c}{With rank correction} \\
    \multicolumn{1}{c}{$\lambda$} & \multicolumn{1|}{c}{$\rho_R$} & \multicolumn{1}{c|}{$\rho_L$} & \multicolumn{1}{c}{$\lambda$} & \multicolumn{1}{c}{$\rho_R$} & \multicolumn{1}{c}{$\rho_L$} \\\hline\rule{0pt}{2.3ex}%
    $-176.67$ &  $2.3\cdot10^{-1}$ & $1.7\cdot10^{-1}$ & $-130.21$ & $1.6\cdot10^{-8}$ & $2.1\cdot10^{-8}$ \\
    $-187.94$ & $1.4\cdot10^{-1}$ & $6.6\cdot10^{-2}$ & $-180.00$ & $5.2\cdot10^{-9}$ & $5.8\cdot10^{-9}$ \\
    $-214.2$ & $1.50\cdot10^{-1}$ & $8.2\cdot10^{-2}$ & $-219.39$ & $3.6\cdot10^{-9}$ & $1.6\cdot10^{-9}$ \\
    \end{tabular}
    \end{center}
\end{table}


We finally ran the FEM model with 33 DOFS.
This leads to a pencil of dimension $35937$ and nrank $k=192$ instead of $72$.
This implies that the augmented problem has dimension $35937\times 2- 192=71682$ and the projected problem dimension 192.
We formed and solved \eqref{eq:PGEP} using RLU with tolerance $10^{-10}$.
We found the correct rank. The computed Ritz values and residual norms can be found in Table~\ref{tab:fem large}.
The condition number of $G_\perp^*(A-\sigma B)F_\perp$ was $3.6\cdot10^3$ where it was $3.0\cdot10^6$ for the randomized approach.
The randomized approach required the construction and orthogonalization of two random $35937\times 192$ matrices, and the projection of the pencil on those.
\begin{table}
    \scriptsize
    \begin{center}
    \begin{tabular}{c|cccc}
& \multicolumn{2}{c}{RLU PGEP} & \multicolumn{2}{c}{randomized PGEP} \\
Ritz value real part & $\rho_R$ & $\rho_L$ & $\rho_R$ & $\rho_L$ \\\hline\rule{0pt}{2.3ex}
    $42.214489$ &  $5.1\cdot10^{-14}$ & $7.5\cdot10^{-14}$ & $1.4\cdot10^{-10}$ & $2.4\cdot10^{-10}$ \\
    $-176.586244$ & $4.3\cdot10^{-14}$ & $4.2\cdot10^{-14}$ & $1.7\cdot10^{-10}$ & $1.5\cdot10^{-10}$ \\
    $-299.251379$ & $1.8\cdot10^{-14}$ & $4.2\cdot10^{-14}$ & $5.4\cdot10^{-12}$ & $1.5\cdot10^{-12}$ \\
    $-310.079052$ & $4.8\cdot10^{-15}$ & $5.3\cdot10^{-15}$ & $1.9\cdot10^{-12}$ & $1.9\cdot10^{-12}$ \\
    $-434.984681$ & $4.5\cdot10^{-15}$ & $1.6\cdot10^{-16}$ & $1.7\cdot10^{-12}$ & $3.5\cdot10^{-12}$ \\
    $-435.008059$ & $5.2\cdot10^{-15}$ & $2.2\cdot10^{-15}$ & $2.0\cdot10^{-12}$ & $3.5\cdot10^{-12}$
     \end{tabular}
        \caption{Comparison for a larger FEM model, with $\rho_R=\|Ax-\lambda Bx\|_2/\|Ax\|_2$, $\rho_L=\|y^* A-\lambda y^* B\|_2/\|y^*A\|_2$}\label{tab:fem large}
    \end{center}
\end{table}

\subsection{Detection of double eigenvalue}\label{sec:double}
We have the parametric matrix $A+\mu B\in\mathbb{C}^{m\times m}$ and we are interested in the values of $\mu$ for which $A+\mu B$ has a double eigenvalue \cite{Elias_DoubleEig}.
Here we set the regularization parameter from \cite{Elias_DoubleEig} to zero and obtain a singular multiparameter eigenvalue problem:
\begin{eqnarray*}
(A + \mu B) x = \lambda x \\
(A + \mu B) y = \lambda y.
\end{eqnarray*}
The elimination of $\lambda$ leads to the singular eigenvalue problem
\[
(A\otimes B - B\otimes A) x = \lambda (I\otimes B - B\otimes I) x,
\]
of dimension $n=m^2$,
with $k=n^2-n$.
Matrices $A$ and $B$ are here chosen as
\[
    A = L_1 \otimes I \quad,\quad
    B = I\otimes L_1 
\]
with $L_1$ the finite difference matrix for the 1D Laplacian on $[0,1]$. The problem represents the discretization of a 2D Laplacian with parameter $\mu$ in the $y$ direction.
We chose $n=100$. For $n=100$, \eqref{eq:AGEP} has dimension $10100$ and \eqref{eq:PGEP} has dimension $9900$. 
We have run two-sided Arnoldi with $m=20$ for \eqref{eq:AGEP} with PLU and RLU factorization and \eqref{eq:PGEP} with PLU, RLU factorization and a randomized selection of $\matV_\perp$ and $\matW_\perp$.
For the factorization $\tau=10^{-10}$ was used. In all cases, the nrank was correctly determined by the LU factorization.
The results are shown in Table~\ref{tab:double}.
These are in line with the results obtained for the previous example.
Note that the randomized \eqref{eq:PGEP} is a dense pencil of dimension $9900$ where the RLU pencil is sparse.

\begin{table}
    \scriptsize
    \begin{center}
    \begin{tabular}{c|cccccc}
      & \multicolumn{2}{c}{PLU PGEP} & \multicolumn{2}{c}{RLU PGEP} & \multicolumn{2}{c}{randomized PGEP} \\    
    Ritz value & $\rho_R$ & $\rho_L$ & $\rho_R$ & $\rho_L$ &  $\rho_R$ & $\rho_L$\\\hline\rule{0pt}{2.3ex}%
$0.66249$ & $9.\cdot10^{-12}$ & $1.0\cdot10^{-11}$ & $3.5\cdot10^{-11}$ & $4.5\cdot10^{-11}$ & $1.4\cdot10^{-7}$ & $5.1\cdot10^{-7}$ \\
$0.66030$ & $7.2\cdot10^{-12}$ & $2.1\cdot10^{-11}$ & $2.6\cdot10^{-11}$ & $1.2\cdot10^{-10}$ & $4.2\cdot10^{-9}$ &  $7.5\cdot10^{-8}$ \\
$\kappa(LU)$ & \multicolumn{2}{c}{$9.1\cdot10^5$} & \multicolumn{2}{c}{$5.4\cdot10^5$} & \multicolumn{2}{c}{$7.9\cdot10^8$} 
\\\hline\rule{0pt}{2.3ex}%
      & \multicolumn{2}{c}{PLU AGEP} & \multicolumn{2}{c}{RLU AGEP} & \multicolumn{2}{c}{randomized AGEP} \\    
    Ritz value & $\rho_R$ & $\rho_L$ & $\rho_R$ & $\rho_L$ \\\hline\rule{0pt}{2.3ex}%
$0.66249$ & $1.8\cdot10^{-10}$ & $4.1\cdot10^{-11}$ & $7.5\cdot10^{-12}$ & $7.9\cdot10^{-12}$ &  $4.2\cdot10^{-9}$ & $6.0\cdot10^{-8}$ \\
$0.66030$ & $1.9\cdot10^{-10}$ & $9.5\cdot10^{-11}$ & $1.3\cdot10^{-11}$ & $8.4\cdot10^{-12}$ & $1.1\cdot10^{-8}$ & $6.9\cdot10^{-8}$ \\
$\kappa(LU)$ & \multicolumn{2}{c}{$9.1\cdot10^5$} & \multicolumn{2}{c}{$4.7\cdot10^5$} & \multicolumn{2}{c}{$1.6\cdot10^9$} \\
    \end{tabular}
        \caption{Comparison for detecting a double eigenvalue with $\rho_R=\|Ax-\lambda Bx\|_2/\|Ax\|_2$, $\rho_L=\|y^* A-\lambda y^* B\|_2/\|y^*A\|_2$}\label{tab:double}
    \end{center}
\end{table}

\subsection{Rectangular eigenvalue problem from nonlinear eigenvalue problem}\label{sec:nonlinear}
Consider the nonlinear eigenvalue problem
\[
A x - \lambda B x + C \frac{r^T x}{s^T x} x = 0
\]
with $A,B,C\in\mathbb{R}^{\ell\times \ell}$ and $r,s\in\mathbb{R}^\ell$ and $x0\neq x\in\mathbb{C}^\ell$ an eigenvector associated with eigenvalue $\lambda\in\mathbb{C}$.
The introduction of the parameter $\mu$ leads to the linear $(\ell+1)\times \ell$ rectangular two-parameter eigenvalue problem
\[
\begin{bmatrix}
    A \\ r^T 
\end{bmatrix} - \lambda \begin{bmatrix}
    B \\ 0
\end{bmatrix} x + \mu \begin{bmatrix}
    C \\ s^T
\end{bmatrix} x.
\]
With $n=(\ell+1)^2$ and $k=\ell^2$, 
the elimination of $\mu$ leads to the rectangular $n\times k$ problem
\begin{equation}\label{eq:pde-right}
\left(\begin{bmatrix}
    A \\ r^T 
\end{bmatrix}\otimes \begin{bmatrix}
    C \\ s^T
\end{bmatrix} -\begin{bmatrix}
    C \\ s^T
\end{bmatrix}\otimes \begin{bmatrix}
    A \\ r^T 
\end{bmatrix}\right)(x\otimes x) = \lambda \left(\begin{bmatrix}
    B \\ r^T 
\end{bmatrix}\otimes \begin{bmatrix}
    C \\ s^T
\end{bmatrix} -\begin{bmatrix}
    C \\ s^T
\end{bmatrix}\otimes \begin{bmatrix}
    B \\ r^T 
\end{bmatrix}\right) (x\otimes x).
\end{equation}
We consider Example~3 from \cite{Claes2022} with $\gamma=2$, and $\ell=100$.
Since the eigenvector has the form $x\otimes x$, the column dimension of the pencil can be reduced from $\ell^2$ to $\ell(\ell+1)/2$ by eliminating unknowns that appear twice in the product $x\otimes x$. This leads to a full rank rectangular $10201\times 5050$ problem.

Since the rank is full, we used PLU with tolerance $\tau=0$.
We applied the single sided Arnoldi method from Theorem~\ref{th:arnoldi-augmented-projected} with subspace dimension $m=20$.
We used shift $\sigma=1$.
For \eqref{eq:AGEP}, we used one implicit restart with zero shift.
The innerproduct from Theorem~\ref{th:arnoldi-augmented-projected} was used.
The Ritz pairs were computed from the Hessenberg matrix of the Arnoldi method.
The residual norms are displayed for the original pencil \eqref{eq:GEP} in Table~\ref{tab:rect}.
Next, instead of using our own LU implementation, we used Octave's built-in LU factorization for sparse matrices and cut the $L$ and $U$ matrices to make them square as discussed in \S\ref{sec:rec}.


We also compared with a randomized method \cite{Hochstenbach2023a} by projecting with an orthonormalized randomly chosen $\matW_\perp$ using Octave's \verb!randn!. Note that computing this projection is not cheap for a large sparse matrix. The current problem at hand is still small size.
This led to a dense matrix, which, for this problems requires $204$MB of storage while the sparse LU requires $98$MB.
The sparse approach required $0.05$ seconds for the factorization, where the randomized approach for computing the projection and factorizing the dense matrix required $2.9$ seconds.
Note that the results of the randomized approach were more accurate than with the linear solver from Octave.

\begin{table}
\caption{Results for the rectangular eigenvalue problem}\label{tab:rect}    
\begin{center}\footnotesize
\begin{tabular}{c|c|c|c|c}
method & AGEP & PGEP & PGEP (octave) & PGEP (randomized) \\\hline\rule{0pt}{2.3ex}
$\kappa(LU)$ & $1.2\cdot10^{13}$ & $6.1\cdot10^{12}$ & $5.5\cdot10^{12}$ & $3.0\cdot10^{12}$ \\
Ritz value  & $0.364246$ & $0.364246$ & $0.364246$ & $0.364246$ \\
$\|Ax-\lambda Bx\|/\|Ax\|$ & $5.9\cdot10^{-8}$ & $1.2\cdot10^{-8}$ & $1.9\cdot10^{-7}$ & $1.9\cdot10^{-9}$ \\
\end{tabular}
\end{center}
\end{table}

\section{Conclusions and outlook}

In this paper, we applied the two-sided shift-and-invert Arnoldi method to regularizations of singular linear pencils, based on the pivoting sequence of a sparse LU factorization.
Standard pivoting approaches (PLU, RLU, CLU) appear to work quite nicely, probably, because the gap of the small and large singular values is not small.
For general problems, only solvers with global maximum volume pivoting are deemed reliable. This is an invitation to further research on practical pivoting strategies for sparse matrices \cite{duffpralet2005}.
Rank correction is introduced.

Since we did not experience an advantage of \eqref{eq:AGEP}, we suggest to use \eqref{eq:PGEP}. 
The randomized approach \cite{Hochstenbach2019} led to a dense matrix and a higher storage cost and computational cost, especially for large scale problems.
We found the approach with orthogonal random $F_\perp$ and $G_\perp$ more accurate than our methods for the NEPv problem, but less accurate for the other problems.

For rectangular problems of full rank, standard sparse factorizations (with partial pivoting) can be used.
Timing comparisons with sparse LU from Octave show a major advantage of sparse methods.

All examples were two-parameter eigenvalue problems. Extensions of the current methods exploiting the Kronecker structure of the matrices is future work.

\section*{Acknowledgement}
The work by Karl Meerbergen is suppported by FWO Research Foundation - Flanders Grant G027624N.
The authors thank Victor Janssens for the formulation of the problem from \S\ref{sec:nonlinear}, and Michiel Hochstenbach and Bor Plestenjak for interesting discussions and suggestions.


\end{document}